\documentclass[11pt,reqno]{amsart}
\usepackage{fullpage,a4wide}

\usepackage[margin=2.7cm]{geometry}

\usepackage{graphicx}
\usepackage[draft]{hyperref}
\usepackage{amsmath,amsopn,amssymb,amsfonts,stmaryrd}
\usepackage{verbatim}
\usepackage{amsthm}
\usepackage{mathtools}
\usepackage{color}
\usepackage{enumitem}
\usepackage[framemethod=TikZ]{mdframed}
\usepackage{bbm}
\usepackage{mathrsfs}
\usepackage{booktabs}
\usepackage{caption}
\usepackage{bm}
\usepackage{tensor}
\usepackage{cleveref}
\newcommand{\IN}{{\mathbb N}}%Natural numbers
\renewcommand{\H}{\mathbb{H}}

\newcommand{\SL}{\mathrm{SL}}

\newcommand{\N}{\mathbb N}
\newcommand{\C}{\mathbb C}
\newcommand{\Q}{\mathbb Q}
\renewcommand{\aa}{a}

\theoremstyle{plain}
\newtheorem{thm}{Theorem}[section]
\newtheorem{cor}[thm]{Corollary}
\newtheorem{lem}[thm]{Lemma}
\newtheorem{prop}[thm]{Proposition}

\theoremstyle{definition}

\numberwithin{equation}{section}

\def\d{\delta}

\def\p{\rho}

\def\n{\nu}

\def\d{\delta}

\def\p{\varrho}

\def\n{\nu}

\newcommand{\re}{{\rm Re}}
\newcommand{\im}{{\rm Im}}

\newcommand{\R}{\mathbb R}
\newcommand{\Z}{\mathbb Z}

\setlist[itemize]{noitemsep, topsep=0pt}

\allowdisplaybreaks

\makeatletter
\newcommand{\vast}{\bBigg@{2}}
\newcommand{\Vast}{\bBigg@{5}}
\makeatother

\renewcommand{\pmod}[1]{\ \left( \mathrm{mod} \, #1 \right)}
\newcommand{\Pmod}[1]{\ ( \mathrm{mod} \, #1 )}

\newcommand{\ord}{\operatorname{ord}}

\DeclareMathOperator{\lcm}{lcm}

\title{Vanishing properties of Fourier coefficients of holomorphic $\eta$-quotients}

\author{Kathrin Bringmann}
\address{University of Cologne, Department of Mathematics and Computer Science, Weyertal 86-90, 50931 Cologne, Germany}
\email{kbringma@math.uni-koeln.de}
\author{Guoniu Han}
\address{I.R.M.A., UMR 7501, Universit\'e de Strasbourg et CNRS, 7 rue
Ren\'e Descartes, F-67084 Strasbourg, France}
\email{guoniu.han@unistra.fr}
\author{Bernhard Heim}
\address{University of Cologne, Department of Mathematics and Computer Science, Weyertal 86-90, 50931 Cologne, Germany}
\email{bheim@uni-koeln.de}
\author{Ben Kane}
\address{The University of Hong Kong, Department of Mathematics, Pokfulam, Hong Kong}
\email{bkane@hku.hk}

\makeatletter
\@namedef{subjclassname@2020}{%
	\textup{2020} Mathematics Subject Classification}
\makeatother
\begin{document}
\date{\today}
	\begin{abstract}
		In this paper, we study vanishing of Fourier coefficients of holomorphic $\eta$-quotients. We investigate examples of two different types: the first one involves integral weight CM newforms, while the second one involves half-integral weight $\eta$-quotients associated with sums of squares and Hurwitz class numbers.
	\end{abstract}
	\maketitle
	\section{Introduction and statement of results}
Let $m\in\N$ and $\delta_j\in \Z$ for $1\leq j\leq m$. We define
	\[
		\prod_{j=1}^{m} \left(q^j;q^j\right)_{\infty}^{\delta_j}=:\sum_{n\ge0}C_{1^{\delta_1} 2^{\delta_2}\cdots m^{\delta_m}}(n)q^n,
	\]
where $(a;q)_n:=\prod_{m=0}^{n-1}(1-aq^m)$ for $n\in\N_0\cup \{\infty\}$ is the \begin{it}$q$-Pochhammer symbol\end{it}. In this paper, we investigate when $C_{1^{\delta_1} 2^{\delta_2}\cdots m^{\delta_m}}(n)$ vanishes. Specifically, define the \begin{it}vanishing set\end{it}
\[
\mathcal{S}_{1^{\delta_1} 2^{\delta_2}\cdots m^{\delta_m}}:=\left\{n\in\N: C_{1^{\delta_1} 2^{\delta_2}\cdots m^{\delta_m}}(n)=0\right\}.
\]
A famous conjecture of Lehmer \cite{Lehmer} states that $\mathcal{S}_{1^{24}}=\emptyset$, while for the partition generating function one sees that $\mathcal{S}_{1^{-1}}=\emptyset$.
Since (see \cite[Theorem 1.60]{O})
\[
	\sum_{n\in\Z}q^{n^2}=\frac{\left(q^2;q^2\right)_{\infty}^{5}}{\left(q;q\right)_{\infty}^2\left(q^4;q^4\right)_{\infty}^{2}},
\]
questions related to whether integers are represented as sums of squares may be interpreted as determinations of $\mathcal{S}_{1^{\delta_1} 2^{\delta_2}\cdots m^{\delta_m}}$. For example, in this language, Lagrange's four-squares theorem, i.e., that every $n\in\N$ may be written in the form $\sum_{j=1}^4 n_j^2=n$ with $n_j\in\Z$ is equivalent to $S_{1^{-8}2^{20}4^{-8}}=\emptyset$.
In another direction, Granville and Ono \cite[Theorem 1]{GranvilleOno} proved that for $t\geq 4$ and $n\in\N$, there always exist a so-called $t$-core partition of $n$, which is equivalent to showing that $S_{1^{-1}t^t}(n)=\emptyset$. To give another interesting example related to $3$-core partitions, in \cite[Theorem 1.1]{HO}, Ono and the second author proved Conjecture 4.6 of \cite{Ha}, which states that\footnote{Here and throughout $p$ denotes a prime.}
\[
\mathcal{S}_{1^8}=\mathcal{S}_{1^{-1}3^{3}}=\left\{n\in\N: \exists p\equiv 2\pmod{3},\ \ord_p(3n+1)\text{ {\rm is odd}}\right\}.
\]
The second author later conjectured that $\mathcal{S}_{1^8} = \mathcal{S}_{1^2 3^2}$ as well, and this was proven by Clemm \cite[Theorem 1 and Remark 2]{Cl14}. This set of three examples is not isolated, as we demonstrate now.
\begin{thm}\label{thm:L52}
We have
\[
\mathcal{S}_{1^{-1}3^34^2}=\mathcal{S}_{1^{4}2^{-2}4^4}=\left\{n\in\N: \exists p\equiv 3\pmod{4},\ \ord_p(3n+2)\text{ {\rm is odd}}\right\}.
\]
\end{thm}
Note that the vanishing set appearing in Theorem \ref{thm:L52} precisely consists of those $n$ for which $3n+2$ is not the norm of an element of $\Z[i]$. In the same way, the vanishing set in the next theorem is related to norms of elements in $\Z[\sqrt{-2}]$.
\begin{thm}\label{thm:L95}
We have
\begin{align*}
\mathcal{S}_{1^{1}2^{-2}4^3}&=\mathcal{S}_{1^{1}2^{2}4^1}=\mathcal{S}_{1^{3}2^{-1}4^2}=\mathcal{S}_{1^32^3}=\mathcal{S}_{1^72^{-3}4^2}\\
&=\left\{n\in\N: \exists p\equiv 5,7\pmod{8},\ \ord_p(8n+3)\text{ {\rm is odd}}\right\}.
\end{align*}
\end{thm}
Like in Theorem \ref{thm:L52}, the vanishing set in the next example is related to norms in $\Z[i]$, but has an additional congruence condition on $n$.
\begin{thm}\label{thm:L65}
We have
\[
\mathcal{S}_{1^{-1}2^{10}3^{-1}4^{-4}}=\mathcal{S}_{1^72^{-2}3^{-1}}=\left\{n\in\N: n\equiv 2\pmod{3}\text{ and }\exists p\equiv 3\pmod{4}, \ord_p(n)\text{ {\rm is odd}}\right\}.
\]
\end{thm}
The behaviour of $\sum_{n\geq 0} C_{1^{\delta_1}\cdots m^{\delta_m}}(n)q^n$ is different depending on the parity of
 $\sum_{j=1}^m \delta_j$.  This is demonstrated by the differing shape of the vanishing sets in the following theorem.
\begin{thm}\label{thm:L133}
We have
\[
\mathcal{S}_{1^22^34^{-2}}=\mathcal{S}_{1^62^{-3}}=\left\{n\in\N: n=4^k(8m+7),\ k,m\in\N_0\right\}.
\]
\end{thm}
The paper is organized as follows: In Section \ref{sec:prelim}, we recall basic facts about modular forms. In Section \ref{sec:S236chi12}, we study the space $S_2(\Gamma_0(36),\chi_{12})$, where $\chi_{D}(n):=(\frac{D}{n})$ with $(\frac\cdot\cdot)$ the extended Legendre symbol. In Section \ref{sec:L52}, we prove \Cref{thm:L52}, in Section \ref{sec:L95} \Cref{thm:L95}, in Section \ref{sec:L65} \Cref{thm:L65}, and in Section \ref{sec:L133} \Cref{thm:L133}.
\section*{Acknowledgments}
The first author has received funding from the European Research Council (ERC) under the European Union’s Horizon 2020 research and innovation programme (grant agreement No. 101001179). The research of the fourth author was supported by grants from the Research Grants Council of the Hong Kong SAR, China (project numbers HKU 17314122 and HKU 17305923).
\section{Preliminaries}\label{sec:prelim}
\subsection{Modular forms}
Here we introduce modular forms, see e.g. \cite{O} for more details. As usual, for $d$ odd, we set
\[
\varepsilon_{d}:=\begin{cases} 1 &\text{if }d\equiv 1\pmod{4}\hspace{-1pt},\\ i&\text{if }d\equiv 3\pmod{4}\hspace{-1pt}.\end{cases}
\]
For $\kappa\in\frac{1}{2}\Z$ and $\gamma=\left(\begin{smallmatrix}a&b\\ c&d\end{smallmatrix}\right)\in\SL_2(\Z)$ ($4\mid c$ if $\kappa\notin \Z$), define the weight $\kappa$ \begin{it}slash operator\end{it} by
\[
f\big|_{\kappa}\gamma(z):=
\begin{cases}
\left( \frac cd \right) \varepsilon_d^{2\kappa}(cz+d)^{-\kappa} f(\gamma z)&\text{if }\kappa\in\Z+\frac{1}{2},\\
(cz+d)^{-\kappa} f(\gamma z)&\text{if }\kappa\in\Z.
\end{cases}
\]
For $\kappa\in\frac{1}{2}\Z$, $N\in\N$ ($4\mid N$ if $\kappa\notin\Z$), and a character $\chi$ (mod $N$), a function $f:\H\to\C$ is a \begin{it}holomorphic modular form of weight $\kappa$ on $\Gamma_0(N)$ with character $\chi$\end{it} if the following conditions hold:
\begin{enumerate}[leftmargin=*]
	\item The function $f$ is holomorphic on $\H$.
	\item We have $f|_{\kappa}\gamma  = \chi(d) f$ for $\gamma\in\Gamma_0(N)$.
	\item For $\gamma\in\SL_2(\Z)$, $%\label{eqn:almostslash}
	(cz+d)^{-\kappa} f(\gamma z)$ is bounded as $z\to i\infty$.
\end{enumerate}
We denote the corresponding space of such forms by $M_{\kappa}\left(\Gamma_0(N),\chi\right)$. We call the equivalence classes of $\Gamma_0(N)\backslash (\Q\cup\{i\infty\})$ the \begin{it}cusps of $\Gamma_0(N)$\end{it}. If the function in (3) vanishes as $z\to i\infty$ for every cusp $\gamma(i\infty)$, then we call $f$ a \begin{it}cusp form\end{it}. We denote the corresponding space by $S_{\kappa}(\Gamma_0(N),\chi)$. We sometimes omit $\chi$ in the notation if it is trivial.

\subsection{Operators on modular forms}\label{sec:operators}
For $f(z)=\sum_{n\in\Z} c(n) q^n$ (with $q:=e^{2\pi i z}$) and $\ell\in\N$, we define the {\it U-operator} and the {\it V-operator} as
\begin{equation*}
 	f \mid U_\ell(z) := \sum_{n\in\Z} c(\ell n) q^n, \quad f \mid V_\ell(z) := f(\ell z).
\end{equation*}
Moreover, we define for $M\in\N$ and $m\in\Z$ the \textit{sieving operator}
\begin{equation*}
 	f \mid S_{M,m}(z) := \sum_{\substack{n\in\Z\\n\equiv m\pmod M}} c(n)q^n.
\end{equation*}
The actions of these operators on half-integral weight modular forms may be found in \cite[Lemma 2.3]{BK}. To state the result, let $\operatorname{rad}(n):=\prod_{p\mid n} p$ be the \begin{it}radical\end{it} of $n\in\N$ and recall that the \begin{it}conductor\end{it} of a character $\chi \pmod N$ is the minimal $N_{\chi}\mid N$ for which there exists a character $\psi \pmod N$ with $\chi(n)=\psi(n)$ for every $n\in\Z$ with $\gcd(n,N)=1$.

\begin{lem}\label{lem:operatorshalf}
Suppose that $f\in M_\kappa(\Gamma_0(N),\chi)$ with $\kappa\in\Z+\frac12$ and $4 \mid N$, and $\chi$ is a character of conductor $N_{\chi}\mid N$.
\noindent

\begin{enumerate}[leftmargin=*, label=\rm(\arabic*)]
\item We have $f|U_{\delta}\in M_\kappa(\Gamma_0(4\lcm(\frac{N}{4},\operatorname{rad}(\delta))),\chi\chi_{4\delta})$.

\item
Suppose that $M\mid 24$ and $M\not\equiv 2\pmod{4}$. Then $f|S_{M,m}\in M_\kappa(\Gamma_{0}(\lcm(N,M^2,MN_{\chi})),\chi)$.

\item
We have $f|V_{\delta}\in M_\kappa(\Gamma_0(N\delta),\chi\chi_{4\delta})$.
\end{enumerate}
\end{lem}
We also require the following lemma for integral-weight modular forms.
\begin{lem}\label{lem:operators}
	Let $N\in\N$, $\chi$ a character (mod $N$) with conductor $N_{\chi}\mid N$, $\kappa\in\N$, and $f\in M_\kappa(\Gamma_{0}(N),\chi)$. Then the following hold:
	\begin{enumerate}[leftmargin=*,label=\rm(\arabic*)]
		\item For $\d\in\N$ we have $f|V_{\d}\in M_\kappa(\Gamma_{0}(N\d),\chi)$.
		\item For $M\in\N$ with $M\mid 24$, $f|S_{M,m}\in M_\kappa(\Gamma_{0}(\lcm(N,M^2,MN_{\chi})),\chi)$.
	\end{enumerate}
\end{lem}
We require the following lemma about products of half-integral weight modular forms.
\begin{lem}\label{lem:halfintmult}
	Suppose $f_1\in M_{\kappa_1+\frac12}(\Gamma_0(N),\psi_1)$, $f_2\in M_{\kappa_2+\frac12}(\Gamma_0(N),\psi_2)$, $f_3\in M_{\kappa_3}(\Gamma_0(N),\psi_3)$ for some $\kappa_1,\kappa_2,\kappa_3\in\N_0$, $N\in\N$, and characters $\psi_1$, $\psi_2$, $\psi_3$ modulo a divisor of $N$. Then $f_1f_2\in M_{\kappa_1+\kappa_2+1}$ $(\Gamma_0(N),\psi_1\psi_2\chi_{-4}^{\kappa_1+\kappa_2+1})$	and $f_1f_3\in M_{\kappa_1+\kappa_3+\frac{1}{2}}(\Gamma_0(N),\psi_1\psi_3\chi_{-4}^{\kappa_3})$.
\end{lem}

For $f_1\in M_{\kappa_1}(\Gamma,\chi_1)$ and $f_1\in M_{\kappa_2}(\Gamma,\chi_2)$, for some group $\Gamma\subseteq\SL_2(\Z)$, and for $\ell\in\IN_0$, the $\ell$-th \begin{it}Rankin--Cohen bracket\end{it} is defined by
\[
\left[f_1,f_2\right]_\ell:=\frac{1}{(2\pi i)^\ell}\sum_{r=0}^\ell \frac{(-1)^r\Gamma(\kappa_1+\ell)\Gamma(\kappa_2+\ell)}{r!(\ell-r)!\Gamma(\kappa_1+r)\Gamma(\kappa_2+\ell-r)} f_1^{(r)}
f_2^{(\kappa-r)}.
\]
By \cite[Corollary 7.2]{Cohen} we have the following.
\begin{lem}\label{lem:RankinCohen}
Suppose that $f_j\in M_{\kappa_1}(\Gamma_0(N),\psi_j)$ ($j\in\{1,2\}$) for some $\kappa_j\in\N_0$, $N\in\N$, and characters $\psi_j\pmod{N}$. Then, for $\ell\in\N_0$, $[f_1,f_2]_{\ell}\in M_{\kappa_1+\kappa_2+2\ell+1}(\Gamma_0(N),\psi_1\psi_2\chi_{-4}^{\kappa_1+\kappa_2+1})$.
\end{lem}

\subsection{Hecke eigenforms}
For $N,\kappa\in\N$, $\chi$ a character (mod $N$), and $p$ a prime we define the \begin{it}Hecke operator\end{it} $T_p$ acting on $f(z)=\sum_{n\ge0}c(n)q^n\in M_\kappa(\Gamma_0(N),\chi)$ by  (see \cite[Definition 2.1]{O})
\[
f\big|T_p(z):=\sum_{n\ge0}\left(c(pn)+\chi(p)p^{\kappa-1}c\left(\frac{n}{p}\right)\right)q^n,
\]
where $c(\alpha):=0$ for $\alpha\in\Q\setminus\N_0$. We call a simultaneous eigenfunction under the Hecke operators $T_p$ for $p\nmid N$ a \begin{it}Hecke eigenform\end{it}. For $\kappa\in\N$, the space $S_\kappa(\Gamma_0(N),\chi)$ splits into the old space, spanned by the images of $V_\d$ on $S_\kappa(\Gamma_0(M),\chi)$ for $\d\mid \frac{N}{M}$ and $M\mid N$ with $M<N$, and the new space is the orthogonal complement of the old space in $S_\kappa(\Gamma_0(N),\chi)$ with respect to the \begin{it}Petersson inner product\end{it} ($z=x+iy$, $f,g\in S_\kappa(\Gamma_0(N),\chi)$)
\[
\left<f,g\right>:=\frac{1}{\left[\SL_2(\Z):\Gamma_0(N)\right]} \int_{\Gamma_0(N)\backslash\H}f(z)\overline{g(z)} y^\kappa \frac{dxdy}{y^2}.
\]
Hecke eigenforms in the new space are called \begin{it}newforms\end{it}. Those newforms whose Fourier expansions $\sum_{n\ge1} c(n)q^n$ have $c(1)=1$ are called  \begin{it}normalized newforms\end{it} (also known as \begin{it}primitive forms\end{it}). We require the following.
\begin{lem}\label{lem:Hecke}
	Let $\kappa,N\in\N$ and $\chi$ a character (mod $N$). Suppose that $f(z)=\sum_{n\geq 1} c(n)q^n$ is a normalized newform in $S_\kappa(\Gamma_0(N),\chi)$. Then the Fourier coefficients $c(n)$ are multiplicative and for every $p\nmid N$ and $r\in\N$ we have
	\[
		c\left(p^{r}\right)= c(p) c\left(p^{r-1}\right) - \chi(p) p^{\kappa-1} c\left(p^{r-2}\right).
	\]
\end{lem}
With $d(n)$ the number of divisors of $n$, Deligne \cite{Deligne} proved the following bound.
\begin{thm}\label{thm:Deligne}
If $f(z)=\sum_{n\geq 1} c(n)q^n$ is a normalized newform in $S_\kappa(\Gamma_0(N),\chi)$, then
\[
	|c(n)| \le d(n) n^\frac{\kappa-1}2.
\]
\end{thm}

\subsection{Eisenstein series}\label{sec:Eisenstein}
Let $\kappa\in\N$ and $\chi,\psi$ primitive characters. We require the modular properties of the Eisenstein series
\[
	E_{\kappa,\chi,\psi}(z):=\mathbbm{1}_{\chi=\chi_1}L(1-\kappa,\psi)+\mathbbm{1}_{\psi=\chi_1}\mathbbm{1}_{\kappa=1}L(0,\chi)+2\sum_{n\geq 1}\sum_{d\mid n}\chi\left(\frac{n}{d}\right)\psi(d) d^{\kappa-1}q^n,
\]
where $L(s,\chi):=\sum_{n\geq 1} \chi(n)n^{-s}$ is defined for $\re(s)>1$ and is meromorphically continued to the whole complex plane. Moreover $\mathbbm{1}_{S}:=1$ if a statement $S$ is true, and $\mathbbm{1}_S:=0$ if $S$ is false. The following modular properties may be found in \cite[Theorem 4.5.1, Theorem 4.6.2, Theorem 4.8.1]{DiamondShurman}.
\begin{lem}\label{lem:Eisenstein}
Suppose that $\kappa,d\in\N$, $\chi$ and $\psi$ are primitive characters of conductors $N_{\chi}$ and $N_{\psi}$, respectively, with $\chi(-1)\psi(-1)=(-1)^\kappa$.
\begin{enumerate}[leftmargin=*,label=\rm(\arabic*)]
\item If $\kappa\neq 2$, then $E_{\kappa,\chi,\psi}\big|V_d \in M_\kappa\left(\Gamma_0\left(N_{\chi}N_{\psi} d\right),\chi\psi\right)$.
\item If $(\chi,\psi)\neq (\chi_1,\chi_1)$, then $E_{2,\chi,\psi}\big|V_d \in M_2\left(\Gamma_0\left(N_{\chi}N_{\psi} d\right),\chi\psi\right)$. If $(\chi,\psi)= (\chi_1,\chi_1)$, then $E_{2,\chi_1,\chi_1}-dE_{2,\chi_1,\chi_1}\big|V_d \in M_2\left(\Gamma_0(d)\right)$.
\end{enumerate}
\end{lem}
\noindent The subspace of modular forms formed by linear combinations of $E_{\kappa,\chi,\psi}\big|V_d$ is the \begin{it}Eisenstein series subspace\end{it}.
 We split a modular form $f=E+g$ where $E$ is contained in the Eisenstein series subspace and $g$ is a cusp form.  We call $E$ the \begin{it}Eisenstein series part\end{it} of $f$ 
and $g$ 
\rm
the \begin{it}cuspidal part\end{it} of $f$.

\subsection{Valence formula}
In order to show identities between modular forms, we use the following lemma, which is a consequence of the valence formula.
\begin{lem}\label{lem:valence}
Let $\kappa \in\frac{1}{2}\N$, $N\in\N$, and $\chi$ be a character (mod $N$). Let $f(z)=\sum_{n\geq0} c(n) q^n\in M_{\kappa}(\Gamma_0(N),\chi)$. If $c(n)=0$ for every $0\leq n\leq N\frac{\kappa}{12}\prod_{p\mid N}(1+\frac{1}{p})$,
then $f\equiv0$.
\end{lem}

\subsection{Modularity of eta-quotients}\label{sec:eta}

Define the \begin{it}Dedekind eta-function\end{it}
\[
\eta(z):=q^{\frac{1}{24}}\prod_{n\ge1}\left(1-q^n\right).
\]

Note
\[
\prod_{j=1}^m \eta(jz)^{\delta_j} =  q^{\frac{1}{24}\sum_{j=1}^m j\delta_j }\prod_{j=1}^m\left(q^j;q^j\right)_{\infty}^{\delta_j}.
\]
Thus, to investigate $\mathcal{S}_{1^{\delta_1}2^{\delta_2}\cdots m^{\delta_m}}$, we require the modularity of certain $\eta$-quotients, which may be found in \cite[Theorem 1.64]{O}.
\begin{lem}\label{lem:eta}
If $f(z)=\prod_{\delta|N}\eta(\delta z)^{r_{\delta}}$ is an $\eta$-quotient of weight $\kappa=\frac12\sum_{\delta\mid N}r_{\delta}\in \Z$, with the additional properties that
\[
\sum_{\delta\mid N} \delta r_{\delta}\equiv 0\pmod{24},\qquad
\sum_{\delta\mid N} \frac{N}{\delta} r_{\delta}\equiv 0\pmod{24},
\]
then $f\in M_\kappa(\Gamma_0(N),\chi_{(-1)^\kappa s})$, where $s:=\prod_{\delta\mid N} \delta^{r_{\delta}}$.
\end{lem}

\subsection{Unary theta functions}

For a character $\chi$ and $j\in\{0,1\}$, as in \cite[Definition 1.42]{O}, we define the \begin{it}unary theta function\end{it}
\[
	\theta(\chi,j,z):=\sum_{n\in\Z} \chi(n) n^jq^{n^2}.
\]
We also let $\Theta(z):=\theta(\chi_1,0,z)=\sum_{n\in\Z}  q^{n^2}$ be the standard {\it theta function} of Jacobi. Recall that $\chi$ is called \begin{it}even\end{it} (resp. \begin{it}odd\end{it}) if $\chi(-1)=1$ (resp. $\chi(-1)=-1$). The modular properties of $\theta(\chi,j,z)$ can be found in \cite[Theorem 1.44]{O}.
\begin{lem}\label{lem:unarytheta}
Suppose that $\chi$ is a primitive Dirichlet character with conductor $N_{\chi}$.
\begin{enumerate}[leftmargin=*,label=\rm(\arabic*)]
\item
If $\chi$ is even, then $\theta(\chi,0,z)\in M_{\frac{1}{2}}(\Gamma_0(4N_{\chi}^2),\chi)$.
\item
If $\chi$ is odd, then $\theta(\chi,1,z)\in S_{\frac{3}{2}}(\Gamma_0(4N_{\chi}^2),\chi\chi_{-4})$.
\end{enumerate}
\end{lem}

\subsection{Binary quadratic forms}
For\footnote{Here and throughout the paper, we use bold letters for vectors.} $a,b\in\N$, let

\rm
\begin{equation*}%\label{eqn:r12def}
	r_{(a,b)}(n):=\#\left\{\bm{n}\in\Z^2:an_1^2+bn_2^2=n\right\}.
\end{equation*}

Jacobi \cite[Proposition 10]{Zag123} obtained the following formula for $r_{(1,1)}(n)$.
\begin{lem}\label{lem:r11n}
	We have, for $n\in\N$,
	\[
		r_{(1,1)}(n)=4\sum_{d\mid n}\left(\frac{-4}{d}\right).
	\]
\end{lem}
A similar expression for $r_{(1,2)}(n)$ is well-known.
\begin{lem}\label{lem:r12n}
	We have
	\[
		\sum_{n\ge0} r_{(1,2)}(n)q^n = 1+2\sum_{n\geq 1} \sum_{d\mid n} \left(\frac{-2}{d}\right) q^n.
	\]
	In particular, $r_{(1,2)}(n)=0$ if and only if there exists $p\equiv 5,7\pmod{8}$ for which $\ord_p(n)$ is odd.
\end{lem}

\subsection{Hurwitz class numbers}
For a discriminant $-D<0$, we let $H(D)$ denote the $D$-th {\it Hurwitz class number}, which counts the number of equivalence classes of positive-definite integral binary quadratic forms of discriminant $-D$, weighted by $\frac{1}{2}$ if the quadratic form is equivalent to a (constant) multiple of $n_1^2+n_2^2$ and weighted by $\frac{1}{3}$ if it is equivalent to a multiple of $n_1^2+n_1n_2+n_2^2$. For $\ell_1,\ell_2\in\N$ with $\gcd(\ell_1,\ell_2)=1$ and $\ell_2$ squarefree, we define
\[
	\mathcal{H}_{\ell_1,\ell_2}:=\mathcal{H}\big|\left(U_{\ell_1\ell_2}-\ell_2U_{\ell_1}\circ V_{\ell_2}\right),
\]
with $\mathcal{H}$ denoting the class number generating function
\[
	\mathcal{H}(z):=\sum_{D\ge0} H(D) q^D.
\]
Using the modularity of $\mathcal{H}$ shown by Zagier \cite{Zagier} (see also \cite[Chapter 2, Theorem 2]{HZ}), the modularity of $\mathcal{H}_{\ell_1,\ell_2}$ was shown in \cite[Lemma 2.6]{BK}. 
\begin{lem}\label{lem:Hell1ell2}
For $\ell_1,\ell_2\in\N$ with $\gcd(\ell_1,\ell_2)=1$ and $\ell_2$ squarefree, we have
\begin{equation*}
	\mathcal{H}_{\ell_1,\ell_2}\in M_\frac32(\Gamma_0(4\operatorname{rad}(\ell_1)\ell_2),\chi_{4\ell_1\ell_2}).
\end{equation*}
\end{lem}
\section{The space $S_2(\Gamma_0(36),\chi_{12})$}\label{sec:S236chi12}
 We require properties of the normalized newforms
\begin{align*}
g_{1}(z)&=q+\sqrt2iq^2-2q^4-\sqrt2i q^5 -2\sqrt2i q^8+O\left(q^{10}\right),\\
g_2(z)&=q-\sqrt2iq^2-2q^4+\sqrt2i q^5 +2\sqrt2i q^8+O\left(q^{10}\right),
\end{align*}
 which generate $S_{2}(\Gamma_0(36),\chi_{12})$. The Fourier coefficients of $g_j(z)=\sum_{n\ge1} c_{j}(n) q^n$  satisfy $c_j(n)\in \Q(\sqrt{-2})$ and $c_2(n)=\overline{c_1(n)}$. Hence
\begin{equation*}
	g_1(z)+g_2(z)=2 \sum_{n\ge1} \re\left( c_1(n)\right) q^n, \quad g_1(z)-g_2(z)=2i \sum_{n\ge1} \im\left(c_1(n)\right) q^n.
\end{equation*}
Using \Cref{lem:operators} (2) and \Cref{lem:valence}, we obtain the following identities.
\begin{lem}\label{L:id}
	We have
	\begin{equation*}%\label{eqn:g1g2sieve}
		\frac{1}{2}\left(g_1+g_2\right) = g_1\big|S_{3,1}, \quad \frac{1}{2}\left(g_1-g_2\right) = g_1\big|S_{3,2}, \quad g_1\big|S_{3,0}=0.
	\end{equation*}
\end{lem}
Using Lemmas \ref{lem:unarytheta} (2), \ref{lem:operatorshalf} (3), \ref{lem:halfintmult}, \ref{lem:valence}, and \cite[Proposition 1.41]{O}, we obtain the following.
\begin{lem}\label{lem:g1g2}
We have, for $j\in\{1,2\}$,
\begin{align*}
g_j(z)&=\frac{(-1)^{j+1}i}{2\sqrt{2}}\left(\Theta(z)-\Theta(9z)\right)\theta(\chi_{-3},1,z) + \frac{1}{2}\theta(\chi_{-3},1,z)\Theta(9z).
\end{align*}
Specifically, we have
\[
c_j(n)=\frac{(-1)^{j+1} i}{2\sqrt{2}}\sum_{\substack{ \bm{n}\in\Z^2, \, 3\nmid n_1\\ n_{1}^{2}+n_{2}^{2}=n}} \chi_{-3}(n_2) n_2  + \frac{1}{2}\sum_{\substack{ \bm{n}\in\Z^2,\, 3\nmid n_1\\ n_1^2+9n_2^2=n}} \chi_{-3}(n_1) n_1.
\]
\end{lem}
Define
\[
\gamma_{1}(n):=\begin{cases}
\frac{c_{1}(n)}{\sqrt{2}i} &\text{if }n\equiv 2\pmod{3},\\
c_{1}(n)&\text{otherwise}.
\end{cases}
\]
A direct calculation using \Cref{lem:g1g2} gives the following.
\begin{lem}\label{lem:gammainteger}
We have $\gamma_{1}(n)\in\Z$, and moreover
\[
	\gamma_{1}(n)=
	\begin{cases}
		\displaystyle\sum_{\substack{\bm{n}\in\N^2\\ n_1^2+n_2^2=n}} \chi_{-3}(n_2)n_2 &\text{if }n\equiv 2\pmod{3},\\
		\mathbbm{1}_{n=\square}\chi_{-3}\left(\sqrt{n}\right)\sqrt{n}+2\displaystyle\sum_{\substack{\bm{n}\in\N^2\\ n_1^2+9n_2^2=n}} \chi_{-3}(n_1)n_1&\text{if }n\equiv 1\pmod{3},\\
		0&\text{if }3\mid n.
	\end{cases}
\]
\end{lem}
We next use Lemma \ref{lem:g1g2} to determine which Fourier coefficients $c_{1}(n)$ vanish.
\begin{prop}\label{prop:f3612vanish}
	We have $c_{1}(n)=0$ if and only if there exists $p\equiv 3\pmod{4}$ for which $\ord_p(n)$ is odd or if $3\mid n$.
\end{prop}
\begin{proof}
Since $g_{1}$ is a newform, it has multiplicative Fourier coefficients by \Cref{lem:Hecke} and the claim is equivalent to showing that $\gamma_1(p^r)=0$ if and only if $p=3$ or ($p\equiv 3\pmod{4}$ and $r$ is odd). Note that the case $p=3$ is already established in Lemma \ref{lem:gammainteger}.

Suppose first that $p\equiv 3\pmod{4}$ and $r$ is odd. Since $n_1^2+n_2^2=p^r$ does not have any integer solutions and $p^r$ is not a square, \Cref{lem:gammainteger} implies that
\begin{equation}\label{eqn:gamma1p3rodd}
\gamma_1(p^r)=0.
\end{equation}

For the reverse direction, we claim that $\gamma_1(2^r)\neq 0$ and if $p\equiv 1\pmod{4}$, then\footnote{Note that since $\gamma_1(p^r)\in\Z$ by Lemma \ref{lem:gammainteger}, \eqref{eqn:gammanotdiv} makes sense as a congruence in the integers.}
\begin{equation}\label{eqn:gammanotdiv}
\gamma_1\left(p^r\right)\not\equiv 0\pmod{p}.
\end{equation}
Note that \eqref{eqn:gammanotdiv} implies that $c_1(p^r)\ne0$ in particular. We prove \eqref{eqn:gammanotdiv} by induction on $r\in\N_0$. Since $\gamma_1(1)=1$ by Lemma \ref{lem:gammainteger}, the claim is true for $r=0$.
In our induction below, we use Hecke relations to relate $\gamma_1(p^r)$ with $\gamma_1(p^{r-1})$ and $\gamma_1(p^{r-2})$, so we need an additional base case $r=1$, which we next prove. First assume that $p=2$. By \Cref{lem:gammainteger}, we have $\gamma_1(2) = 1$.

Next suppose that $p\equiv 1\pmod{4}$. By Lemma \ref{lem:r11n}, we have $8$ solutions in $\Z^2$ to $n_{1}^{2}+n_{2}^{2}=p$. Fixing one solution $\bm{a}\in\N^2$, we obtain the 8 solutions by\footnote{Note that $a_1=a_2$ implies $p=a_1^2+a_2^2=2a_2^2$, which contradicts $p\equiv 1\pmod{4}$.} $(\pm a_1,\pm a_2)$ and $(\pm a_2,\pm a_1)$. Thus $\bm{n}\in\N^2$ satisfies $n_1^2+n_2^2=p$ if and only if
\begin{equation}\label{eqn:ninset}
\bm{n}\in \{\bm{a},(a_2,a_1)\}.
\end{equation}

If $p\equiv 5\pmod{12}$, then by \eqref{eqn:ninset} there are precisely two terms $\bm{n}=\bm{a}$ and $\bm{n}=(a_2,a_1)$ in the sum in Lemma \ref{lem:gammainteger} and thus
\[
\gamma_1(p)=\chi_{-3}\left(a_1\right)a_1+\chi_{-3}\left(a_2\right)a_2.
\]
Since $a_1^2+a_2^2=p$, we have $a_j\leq \sqrt{p}$ and $p\ge5$ implies that $p>2\sqrt{p}$, so $|\gamma_1(p)|<p$. Since $\gamma_1(p)=0$ is impossible, $\gamma_1(p)\not\equiv 0\pmod{p}$.

If $p\equiv 1\pmod{12}$, then exactly one of $a_1$ or $a_2$ is divisible by $3$. Without loss of generality, assume that $3\mid a_2$.  Writing the terms in the sum from Lemma \ref{lem:gammainteger} as $n_1^2+(3n_2)^2=p$,
we see from \eqref{eqn:ninset} that $(n_1,3n_2)\in\{\bm{a},(a_2,a_1)\}$. Since $3\mid a_2$, we see that the sum has a single term $\bm{n}=(a_1,\frac{a_2}{3})$ and
\[
\gamma_1(p)=2\chi_{-3}(a_1)a_1.
\]
Then $\frac12|\gamma_1(p)|<p$. Since $a_1\ne0$, we have $\gamma_1(p)\neq 0$.  Since $\gamma_1(p)\neq 0$ and $\frac12|\gamma_1(p)|<p$, we see that \eqref{eqn:gammanotdiv} holds in this case as well. This completes the case $r=1$ of \eqref{eqn:gammanotdiv}.

Since $g_1\in S_2(\Gamma_0(36),\chi_{12})$ is a normalized newform, \Cref{lem:Hecke} implies that
	\begin{equation}\label{eqn:Heckerel}
		c_1\left(p^r\right) = c_1(p) c_1\left(p^{r-1}\right) - \chi_{12}(p)pc_1\left(p^{r-2}\right).
	\end{equation}
For $p=2$,
we have $\gamma_1(2)=1$ by \Cref{lem:gammainteger}, and \eqref{eqn:Heckerel} implies that
\[
\gamma_1(2^r)=
\begin{cases}
\gamma_1(2)\gamma_1\left(2^{r-1}\right)&\text{if $r$ is odd},\\
-2\gamma_1(2)\gamma_1\left(2^{r-1}\right)&\text{if $r$ is even}.
\end{cases}
\]
Hence for $r\in\N$ we have $\gamma_1(2^r)=(-2)^{\lfloor\frac{r}{2}\rfloor}\neq 0$ by induction.

Next suppose that $p\equiv 1\pmod{4}$ and assume that \eqref{eqn:gammanotdiv} holds for $j\in\N$ with $j<r$.
 If $p\equiv 1\pmod{12}$, then $p^j\equiv 1\pmod{3}$ for all $j\in\N_0$, so $\gamma_1(p^j)=c_1(p^j)$ and \eqref{eqn:Heckerel} implies
\begin{equation*}
\gamma_1\left(p^r\right)=\gamma_1(p) \gamma_1\left(p^{r-1}\right) - \chi_{12}(p)p\gamma_1\left(p^{r-2}\right)\equiv \gamma_1(p) \gamma_1\left(p^{r-1}\right) \not \equiv 0\pmod{p},
\end{equation*}
where we use the inductive hypothesis \eqref{eqn:gammanotdiv} and $\gamma_1(p) \not\equiv 0 \pmod{p}$ in the last step.

For $p\equiv 5\pmod{12}$, we have $p^r\equiv 1\pmod{3}$ if $r$ is even and $p^r\equiv 2\pmod{3}$ if $r$ is odd, so we split into the cases $r$ even and $r$ odd. For $r$ even, we have $r-1$ odd and $r-2$ even, so \eqref{eqn:Heckerel} implies that
\begin{align*}
\gamma_1\left(p^r\right)= \sqrt{2}i\gamma_1(p)\sqrt{2}i\gamma_1\left(p^{r-1}\right) -\chi_{12}(p)p\gamma_1\left(p^{r-2}\right) \equiv  -2\gamma_1(p)\gamma_1\left(p^{r-1}\right)\not\equiv 0\pmod{p},
\end{align*}
where we use the inductive hypothesis, $p\neq 2$, and $\gamma_1(p) \not\equiv 0 \pmod{p}$ in the last step.

If $p\equiv 5\pmod{12}$ and $r$ is odd, then $r-1$ is even and $r-2$ is odd, so \eqref{eqn:Heckerel} implies that
\begin{equation*}
\gamma_1\left(p^r\right)\equiv\gamma_1(p)\gamma_1\left(p^{r-1}\right)\not\equiv 0\pmod{p}.
\end{equation*}

We finally inductively show that $c(p^r)\neq 0$ for $3<p\equiv 3\pmod{4}$ and $r$ even. The base case $r=0$ is established by $c_1(1)=1$. Suppose that $r\geq 2$ is even. Since $r-1$ is odd, we have $c_1(p^{r-1})=0$ by \eqref{eqn:gamma1p3rodd}. Hence in this case \eqref{eqn:Heckerel} simplifies as
\[
c_1\left(p^r\right)= - \chi_{12}(p)pc_1\left(p^{r-2}\right)\neq 0
\]
by induction. In the last step, we use the fact that $\chi_{12}(p)\neq 0$ because $p\neq 3$.
\rm
 \end{proof}

\section{Proof of \Cref{thm:L52}}\label{sec:L52}
\subsection{The case $1^{-1}3^34^2$}
In this subsection, we prove half of \Cref{thm:L52}.
\begin{prop}\label{prop:1^{-1}3^34^2}
We have
\[
\mathcal{S}_{1^{-1}3^34^2}
%=\mathcal{S}_{1^{4}2^{-2}4^4}
=\left\{n\in\N: \exists p\equiv 3\pmod{4},\ \ord_p(3n+2)\text{ {\rm is odd}}\right\}.
\]
\end{prop}

Using Lemmas \ref{lem:operators} (2), \ref{lem:eta}, and \ref{lem:valence}, we first relate the eta-quotient to $g_1$ and the newform
\[
g_{3}(z)=q+3\sqrt2i q^5+4q^{13}-3\sqrt2i q^{17}+O\left(q^{25}\right)\in S_{2}\left(\Gamma_0(144),\chi_{12}\right).
\]
\begin{lem}\label{lem:(1,-1),(3,3),(4,2)}
We have
\[
	\frac{\eta^3(9z)\eta^2(12z)}{\eta(3z)}=  -\frac i{\sqrt2} g_{1}\big|S_{12,2}(z) + \frac{i}{\sqrt{2}} g_{1}\big|S_{12,8}(z) - \frac i{3\sqrt2} g_{3}\big|S_{6,5}(z).
\]
\end{lem}
In addition to Lemma \ref{lem:gammainteger}, we require a formula for the Fourier coefficients $c_3(n)$ of $g_{3}$.
Using Lemmas \ref{lem:unarytheta}, \ref{lem:operatorshalf} (3), \ref{lem:halfintmult}, \ref{lem:operators} (2), and \ref{lem:valence}, it is not hard to show the following formula for $g_3$ and its Fourier coefficients.
\begin{lem}\label{lem:g521}
We have
\begin{align*}
g_{3}(z)=&\left(\theta\left(\chi_{-3},1,4z\right)\Theta(z)\right)\big|S_{12,1} +\frac{1}{2} \left(\theta\left(\chi_{-3},1,z\right)\Theta(4z)\right)\big|S_{12,1}\\
&+\frac{i}{\sqrt{2}}\left(\theta\left(\chi_{-3},1,4z\right)\Theta(z)\right)\big|S_{12,5} +\frac{i}{2\sqrt{2}} \left(\theta\left(\chi_{-3},1,z\right)\Theta(4z)\right)\big|S_{12,5}.
\end{align*}
In particular, the $n$-th Fourier coefficient $c_3(n)$ of $g_{3}$ is
\[
\begin{cases}
\mathbbm{1}_{n=\square}\chi_{-3}(\sqrt{n}) \sqrt{n}+ 4\hspace{-.25cm}\displaystyle\sum_{\substack{\bm{m}\in\N^2\\ 4m_1^2+9m_2^2=n}}\hspace{-.25cm} \chi_{-3}\left(m_1\right)m_1 + 2\hspace{-.25cm}\sum_{\substack{\bm{m}\in\N^2\\ m_1^2+36m_2^2=n}}\hspace{-.25cm} \chi_{-3}\left(m_1\right)m_1&\text{if }n\equiv 1\pmod{12},\\
2\sqrt{2}i\hspace{-.25cm}\displaystyle\sum_{\substack{\bm{m}\in\N^2\\ 4m_1^2+m_2^2=n}}\hspace{-.25cm} \chi_{-3}\left(m_1\right)m_1 + \sqrt{2}i\hspace{-.25cm}\sum_{\substack{\bm{m}\in\N^2\\ m_1^2+4m_2^2=n}}\hspace{-.25cm} \chi_{-3}\left(m_1\right)m_1&\text{if }n\equiv 5\pmod{12},\\
0&\text{otherwise}.
\end{cases}
\]
\end{lem}
A computation similar to the case of \Cref{prop:f3612vanish} gives a classification for those $n\in\N$ for which $c_3(n)$ vanishes.
\begin{prop}\label{prop:g521vanish}
	For $n\in\N$, we have that $c_{3}(n)$ vanishes if and only if $\gcd(n,6)>1$ or if there exists a prime $p\equiv 3\pmod{4}$ with $\ord_p(n)$ odd.
\end{prop}

Lemma \ref{lem:(1,-1),(3,3),(4,2)} and \Cref{prop:g521vanish} now directly imply \Cref{prop:1^{-1}3^34^2}.

\subsection{The case $1^42^{-2}4^4$}
In this subsection, we prove the other half of \Cref{thm:L52}.
\begin{prop}\label{prop:1^42^{-2}4^4}
We have
\[
\mathcal{S}_{1^{4}2^{-2}4^4}=\left\{n\in\N: \exists p\equiv 3\pmod{4},\ \ord_p(3n+2)\text{ {\rm is odd}}\right\}.
\]
\end{prop}
We first relate the eta-quotient to a normalized newform
\[
g_4(z)=q-2q^2+4q^4+8q^5-8q^8+O\left(q^{10}\right)\in S_{3}\left(\Gamma_0(36),\chi_{-4}\right).
\]
Using Lemmas \ref{lem:eta}, \ref{lem:operators}, and \ref{lem:valence}, we obtain the following.
\begin{lem}\label{lem:142m244}
	We have
	\[
		\frac{\eta^4(3z)\eta^4(12z)}{\eta^2(6z)}=-\frac{1}{2}g_{4}\big|S_{3,2}(z).
	\]
\end{lem}
We find the following formulas for $g_{4}$, letting
\begin{equation*}
	\Theta_1(z):=2\sum_{\bm{n}\in\Z^2} \left(\frac{-3}{n_1n_2}\right) n_1n_2 q^{n_1^2+n_2^2}, \quad \Theta_2(z):=\sum_{\bm{n}\in\Z^2} n_1^2 q^{n_1^2+9n_2^2}, \quad \Theta_3(z):=9\sum_{\bm{n}\in\Z^2} n_2^2 q^{n_1^2+9n_2^2}.
\end{equation*}

\begin{lem}\label{lem:f522}
We have
\begin{equation*}
	g_{4}=-\frac{\Theta_1}4+\frac{\Theta_2}2-\frac{\Theta_3}2,\quad g_{4}\big|S_{3,0}=0, \quad g_{4}\big|S_{3,1}=\frac{\Theta_2}2-\frac{\Theta_3}2, \quad g_{4}\big|S_{3,2}=-\frac{\Theta_1}4.
\end{equation*}
Moreover
\[
	c_{4}(n)=
	\begin{cases}
		\mathbbm{1}_{n=\square} n + 2\sum_{\substack{\bm{n}\in \N^2\\ n_1^2+9n_2^2=n}} \left(n_1^2-9n_2^2\right)&\text{if }n\equiv 1\pmod{3},\vspace{0.1cm}\\
		-2\sum_{\substack{\bm{n}\in\N^2\\ n_1^2+n_2^2=n}} \left(\frac{-3}{n_1n_2}\right) n_1n_2&\text{if }n\equiv 2\pmod{3},\\
		0&\text{if }n\equiv 0\pmod{3}.
	\end{cases}
\]
\end{lem}
\begin{proof}
	Using Lemma \ref{lem:operators} (2) the left-hand sides of the first four claimed identities are weight $3$ cusp forms on $\Gamma_0(36)$ with character $\chi_{-4}$.
	We next write $\Theta_1(z) =2\theta^2(\chi_{-3},1,z)$. Using Lemmas \ref{lem:unarytheta} (2) and \ref{lem:halfintmult}, we conclude that $\Theta_1\in S_3\left(\Gamma_0(36),\chi_{-4}\right)$.

By \Cref{lem:operatorshalf} (3) and \Cref{lem:RankinCohen}, we have
	\begin{equation*}
		\Theta_2 - \Theta_3 = 2\left[\Theta,\Theta\mid V_9\right]_1\in M_{3}\left(\Gamma_0(36),\chi_{-4}\right).
	\end{equation*}
	Thus we conclude that the right-hand sides of the first four identities are in $M_3(\Gamma_0(36),\chi_{-4})$. \Cref{lem:valence} then gives the identities for the modular form, the identity for $c_{4}(n)$ following by a direct calculation, picking off the Fourier coefficients.
\end{proof}
To prove \Cref{prop:1^42^{-2}4^4} we require the following proposition; its proof is similar to that of \Cref{prop:f3612vanish}.
\begin{prop}\label{prop:f522vanish}
	We have $c_{4}(n)=0$ if and only if $3\mid n$ or there exists a prime $p\equiv 3\pmod{4}$ for which $\ord_p(n)$ is odd.
\end{prop}
\Cref{lem:142m244} and \Cref{prop:f522vanish} now directly imply \Cref{prop:1^42^{-2}4^4}. %We are now ready to prove \Cref{thm:L52}.
\section{Proof of \Cref{thm:L95}}\label{sec:L95}
\subsection{The case $1^12^{-2}4^3$}\label{sec:1^12^{-2}4^3}
The goal of this subsection is the claimed evaluation of the first set appearing in \Cref{thm:L95}.
\begin{prop}\label{prop:1^12^{-2}4^3}
We have
\begin{equation*}
	\mathcal{S}_{1^{1}2^{-2}4^3}=\left\{n\in\N: \exists p\equiv 5,7\pmod{8},\ \ord_p(8n+3)\ \text{\rm is odd}\right\}.
\end{equation*}
\end{prop}
We first relate $C_{1^1,2^{-2},4^3}(n)$ to $r_{1,2}(n)$. Using Lemmas \ref{lem:eta}, \ref{lem:unarytheta} (1), \ref{lem:halfintmult}, \ref{lem:operatorshalf} (3), \ref{lem:valence}, and \cite[Proposition 1.41]{O} yields the following.
\begin{lem}\label{lem:(1,1),(2,-2),(4,3)}
	We have
	\[
		\frac{\eta(z)\eta^3(4z)}{\eta^2(2z)} = \frac14\sum_{n\ge0} (-1)^n r_{(1,2)}(8n+3) q^{n+\frac{3}{8}}.
	\]
\end{lem}
\Cref{prop:1^12^{-2}4^3} now directly follows from Lemmas \ref{lem:(1,1),(2,-2),(4,3)} and \ref{lem:r12n}.

\subsection{The case $1^12^24^1$}
The goal of this subsection is the claimed evaluation of the second set appearing in \Cref{thm:L95}.
\begin{prop}\label{prop:1^12^24^1}
We have
\[
\mathcal{S}_{1^{1}2^{2}4^1}=\left\{n\in\N: \exists p\equiv 5,7\pmod{8},\ \ord_p(8n+3)\text{ {\rm is odd}}\right\}.
\]
\end{prop}
We first relate the eta-quotient to CM newforms (see \cite[Section 1]{Schutt} for a definition). Using Lemmas \ref{lem:unarytheta}, \ref{lem:operatorshalf} (3), \ref{lem:halfintmult}, \ref{lem:operators} (2), \ref{lem:eta}, and \ref{lem:valence} gives the following by a direct calculation.
\begin{lem}\label{lem:(1,1),(2,2),(4,1)}
We have
\[
\eta(8z)\eta^2(16z)\eta(32z)=\frac{1}{2\sqrt{2}} g_{5}\big|S_{8,3},
\]
where
\[
g_{5}(z):= \frac{1}{2}\theta\left(\chi_{-8},1,z\right)\theta\left(\chi_1,0,8z\right) + \frac{1}{\sqrt{2}} \theta\left(\chi_{8},0,z\right)\theta\left(\chi_{-4},1,2z\right)\in S_{2}\left(\Gamma_0(256)\right)
\]
is a normalized newform. Moreover, we have
\[
c_{5}(n)=
\begin{cases}
\mathbbm{1}_{n=\square}\left(\frac{-2}{\sqrt n}\right) \sqrt n +  2\sum_{\substack{\bm{n}\in\N^2\\ n_1^2+8n_2^2=n}} \left(\frac{-2}{n_1}\right)n_1&\text{if }n\equiv 1\pmod{8},\vspace{0.1cm}\\
2\sqrt{2}\sum_{\substack{\bm{n}\in\N^2\\ n_1^2+2n_2^2=n}} \left(\frac{2}{n_1}\right) \left(\frac{-4}{n_2}\right) n_2&\text{if }n\equiv 3\pmod{8},\\
0&\text{otherwise}.
\end{cases}
\]
\end{lem}
\Cref{prop:1^12^24^1} now follows similar to \Cref{prop:f522vanish}.

\subsection{The case $1^32^{-1}4^2$}
The goal of this subsection is the claimed evaluation of the third set appearing in \Cref{thm:L95}.
\begin{prop}\label{prop:1^32^{-1}4^2}
We have
\[
\mathcal{S}_{1^{3}2^{-1}4^2}=\left\{n\in\N: \exists p\equiv 5,7\pmod{8},\ \ord_p(8n+3)\text{ {\rm is odd}}\right\}.
\]
\end{prop}

We first use Lemmas \ref{lem:eta}, \ref{lem:operators} (3), and \ref{lem:valence} to relate the eta-quotient to a normalized newform
\[
g_6(z)=q+2iq^3-q^9+O\left(q^{10}\right) \in S_2(\Gamma_0(64),\chi_8).
\]
\begin{lem}\label{lem:132m142}
We have
\begin{equation*}%\label{eqn:eta953}
	\frac{\eta^3(8z)\eta^2(32z)}{\eta(16z)}=-\frac{i}{2} g_{6}\big|S_{8,3}(z).
\end{equation*}
\end{lem}

We next compute the Fourier coefficients $c_{6}(n)$ of $g_{6}$. Using Lemmas \ref{lem:unarytheta}, \ref{lem:operatorshalf} (3), \ref{lem:halfintmult}, \ref{lem:operators} (2), and \ref{lem:valence} directly yields the following.
\begin{lem}\label{lem:f953}
We have
\[
g_{6}(z)=\frac{1}{2}\left(\theta\left(\chi_{-4},1,z\right)\Theta(2z)\right)\big|S_{8,1} + \frac{i}{2}\left(\theta\left(\chi_{-4},1,z\right)\Theta(2z)\right)\big|S_{8,3}.
\]
Moreover, we have
\[
c_{6}(n)=\begin{cases}
\mathbbm{1}_{n=\square}\left(\frac{-1}{\sqrt{n}}\right)\sqrt n +2\sum_{\substack{\bm{m}\in\N^2\\ m_1^2+2m_2^2=n}} \left(\frac{-1}{m_1}\right) m_1 &\text{if }n\equiv 1\pmod{8},\\
2i\sum_{\substack{\bm{m}\in\N^2\\ m_1^2+2m_2^2=n}} \left(\frac{-1}{m_1}\right) m_1 &\text{if }n\equiv 3\pmod{8}.
\end{cases}
\]
\end{lem}

\Cref{prop:1^32^{-1}4^2} now follows similarly to \Cref{prop:f522vanish}.
\rm
\rm

\subsection{The case $1^32^{3}$}
The goal of this subsection is the claimed evaluation of the fourth set appearing in \Cref{thm:L95}.
\begin{prop}\label{prop:1^32^3}
We have
\[
\mathcal{S}_{1^32^3}=\left\{n\in\N: \exists p\equiv 5,7\pmod{8},\ \ord_p(8n+3)\text{ {\rm is odd}}\right\}.
\]
\end{prop}

Using Lemmas \ref{lem:eta}, \ref{lem:operators} (2), and \ref{lem:valence}, we relate the eta-quotient to a normalized newform
\[
g_7(z)=q+4\sqrt{2} q^3+23q^{9}+ O\left(q^{10}\right)\in S_3(\Gamma_0(128),\chi_{-8}).
\]
\begin{lem}\label{lem:1323}
	We have
	\[
		\eta^3(8z)\eta^3(16z)=\frac{1}{4\sqrt{2}}g_{7}\big|S_{8,3}(z).
	\]
\end{lem}

We next find a formula for the Fourier coefficients of $g_{7}$. To state the formula, set
\begin{equation*}
\Theta_4(z):=\sum_{\bm{m}\in\Z^2} \left(\frac{-4}{m_1m_2}\right)m_1m_2 q^{m_1^2+2m_2^2}, \quad \Theta_5(z):=\sum_{\substack{\bm{m}\in\Z^2\\2\nmid m_1}} (-1)^{m_2}\left(m_1^2-8m_2^2\right) q^{m_1^2+8m_2^2}.
\end{equation*}
A direct calculation using \cite[Theorem 1.60]{O}, \Cref{lem:eta}, \cite[Theorem 1.41]{O}, and \Cref{lem:valence} gives the following.
\begin{lem}\label{lem:f954}
	We have
	\[
		g_{7}=\sqrt{2} \Theta_4+ \frac{\Theta_5}2.
	\]
In particular, we have
\[
c_{7}(n)=
\begin{cases}
	\mathbbm{1}_{n=\square} n +2\sum_{\substack{\bm{m}\in\N^2\\ m_1^2+8m_2^2=n}}(-1)^{m_2}\left(m_1^2-8m_2^2\right)&\text{if }n\equiv 1\pmod{8},\\[-8pt]\\
	4\sqrt{2}\sum_{\substack{\bm{m}\in\N^2\\ m_1^2+2m_2^2=n}} \left(\frac{-1}{m_1m_2}\right)m_1m_2&\text{if }n\equiv 3\pmod{8},\\
	0&\text{otherwise.}
\end{cases}
\]
\end{lem}

\Cref{prop:1^32^3} now follows similar to \Cref{prop:f3612vanish}.
\rm

\subsection{The case $1^72^{-3}4^2$}
The goal of this subsection is the claimed evaluation of the fifth set appearing in \Cref{thm:L95}.
\begin{prop}\label{prop:1^72^{-3}4^2}
We have
\[
\mathcal{S}_{1^72^{-3}4^2}=\left\{n\in\N: \exists p\equiv 5,7\pmod{8},\ \ord_p(8n+3)\text{ {\rm is odd}}\right\}.
\]
\end{prop}

Using Lemmas \ref{lem:eta}, \ref{lem:operators} (2), and \ref{lem:valence}, we relate the eta-quotient to a normalized newform
\[
g_8(z)=q-2q^3-5q^9+O\left(q^{10}\right)\in S_3(\Gamma_0(32),\chi_{-8}).
\]
\begin{lem}\label{lem:(1,7),(2,-3),(4,2)}
	We have
	\[
		\frac{\eta^7(8z)\eta^2(32z)}{\eta^3(16z)}=\frac12g_8\big|S_{8,3}(z).
	\]
\end{lem}
We next obtain the following formula for the Fourier coefficients of $g_{8}$, using Lemmas \ref{lem:f954}, \ref{lem:operatorshalf} (3), and \ref{lem:valence}, setting
\begin{equation*}
	\Theta_6(z):=\sum_{\substack{\bm{n}\in\Z^2\\2\nmid n_1}} \left(n_1^2-8n_2^2\right) q^{n_1^2+8n_2^2}, \quad
	\Theta_7(z):=\sum_{\substack{\bm{n}\in\Z^2\\2\nmid n_1,n_2}} \left(n_1^2-2n_2^2\right) q^{n_1^2+2n_2^2},
\end{equation*}
\begin{lem}\label{lem:f955}
	We have
	\[
		g_{8}=\frac{\Theta_6}2-\frac{\Theta_7}2.
	\]
	In particular,
	\[
		c_{8}(n)=\begin{cases}
		\mathbbm{1}_{n=\square} n + 2\sum_{\substack{\bm{m}\in\N^2\\ m_1^2+8m_2^2=n}} \left(m_1^2-8m_2^2\right)&\text{if }n\equiv 1\pmod{8},\vspace{0.1cm}\\
		-2\sum_{\substack{\bm{m}\in\N^2\\ m_1^2+2m_2^2=n}} \left(m_1^2-2m_2^2\right)&\text{if }n\equiv 3\pmod{8},\\
		0&\text{otherwise}.
	\end{cases}
	\]
\end{lem}
\Cref{prop:1^72^{-3}4^2} now follows similarly to \Cref{prop:f3612vanish}.

\section{Proof of \Cref{thm:L65}}\label{sec:L65}

\subsection{The case $1^{-1}2^{10}3^{-1}4^{-4}$}\hspace{0cm}
Define
\[
	f_1(z):=\frac{\eta^{10}(2z)}{\eta(z)\eta(3z)\eta^4(4z)}.
\]
Since $f_1$ is not a cusp form, it has a non-trivial Eisenstein series part. In order to investigate the Eisenstein series part, we define
\begin{multline*}\label{E:Eis}
	\mathcal E_1(z):= 1
	+\sum_{n\ge1}\left( \frac{1}{2}\sum_{d\mid n} \left(\frac{12}{d}\right)d +\frac{1}{2}\sum_{d\mid n} \left(\frac{-3}{\frac{n}{d}}\right)\left(\frac{-4}{d}\right) d +2 \sum_{d\mid n} \left(\frac{-4}{\frac{n}{d}}\right)\left(\frac{-3}{d}\right)d +2\sum_{d\mid n} \left(\frac{12}{\frac{n}{d}}\right)d\right.\nonumber\\
	  \left.- \frac{3}{2}\sum_{d\mid \frac{n}{3}} \left(\frac{12}{d}\right)d +\frac{9}{2}\sum_{d\mid \frac{n}{3}} \left(\frac{-3}{\frac{n}{3d}}\right)\left(\frac{-4}{d}\right) d +6 \sum_{d\mid \frac{n}{3}} \left(\frac{-4}{\frac{n}{3d}}\right)\left(\frac{-3}{d}\right)d -18\sum_{d\mid \frac{n}{3}} \left(\frac{12}{\frac{n}{3d}}\right)d \right)q^n.\nonumber
\end{multline*}
A direct calculation gives the following identity for $\mathcal E_1$ in terms of Eisenstein series.
\begin{lem}\label{lem:E651}
	We have
	\begin{multline*}
		\mathcal E_1=\frac{1}{4}E_{2,\chi_1,\chi_{12}}+\frac{1}{4}E_{2,\chi_{-3},\chi_{-4}}+E_{2,\chi_{-4},\chi_{-3}}+E_{2,\chi_{12},\chi_{1}}-\frac{3}{4}E_{2,\chi_1,\chi_{12}}\big|V_3+\frac{9}{4}E_{2,\chi_{-3},\chi_{-4}}\big|V_3\\%\label{eqn:E651}\\
		+3E_{2,\chi_{-4},\chi_{-3}}\big|V_3 - 9E_{2,\chi_{12},\chi_1}\big|V_3\in M_{2}\left(\Gamma_0(36),\chi_{12}\right).
	\end{multline*}
\end{lem}
Recall the newforms $g_1$ and $g_2$, defined in Section \ref{sec:S236chi12}, which span the space $S_2(\Gamma_0(36),\chi_{12})$. Using Lemmas \ref{lem:E651} and \ref{lem:valence} we obtain an identity for $f_1$.
\begin{lem}\label{lem:f65identity}
	We have
	\begin{equation*}%\label{eqn:f65identity}
		f_1 = \mathcal E_1 - 2\Big(1+\sqrt2i\Big)g_1 - 2\Big(1-\sqrt2i\Big)g_2.
	\end{equation*}
\end{lem}

We first rewrite the Fourier coefficients of $\mathcal E_1$. A direct calculation gives the following.
\begin{lem}\label{lem:E65coeff}
Suppose that $\nu_2,\nu_3\in\N_0$ and $m\in\N$ with $\gcd(m,6)=1$. Then the $2^{\nu_2}3^{\nu_3}m$-th Fourier coefficient of $\mathcal E_1$ is
\begin{multline*}
\bigg(\frac12 + (-1)^{\nu_2+\nu_3}\frac{3^{\nu_3}}{2} \left(\frac{-3}{ m}\right) +(-1)^{\nu_2+\nu_3} 2^{\nu_2+1}\left(\frac{-1}{m}\right) + 2^{\nu_2+1} 3^{\nu_3} \left(\frac3m\right)\\
	- \mathbbm{1}_{3\mid n}\left(\frac32 +(-1)^{\nu_2+\nu_3} \frac{3^{\nu_3+1}}2  \left(\frac{-3}{m}\right)  +  (-1)^{\nu_2+\nu_3}3\cdot 2^{\nu_2+1}\left(\frac{-1}{m}\right) + 2^{\nu_2+1} 3^{\nu_3+1} \left(\frac3m\right)\right)\bigg)\\
\times \prod_{p\mid m}\frac{1-\left(\left(\frac3p\right)p\right)^{\ord_p(m)+1}}{1-\left(\frac3p\right)p}.
\end{multline*}
\end{lem}
We next use Lemma \ref{lem:E65coeff} to show the following.
\begin{cor}\label{cor:f651-not2mod3}
For $n\not\equiv 2\pmod{3}$, the $n$-th Fourier coefficient of $f_1$ does not vanish.
\end{cor}
\begin{proof}
	We first bound the cuspidal part of $f_1$. Recalling that $c_j(n)$ denotes the $n$-th Fourier coefficient of $g_j$, the $n$-th Fourier coefficient of the cuspidal part of $f_1$ is
	\begin{equation}\label{E:c1c2sum}
		-2(c_1(n)+c_2(n)) +2\sqrt{2}i(c_2(n)-c_1(n)).
	\end{equation}
	By \Cref{L:id}, we have $c_1(n)+c_2(n)=0$ unless $n\equiv 1\pmod{3}$ and $c_2(n)-c_1(n)=0$ unless $n\equiv 2\pmod{3}$. Thus we conclude from \Cref{thm:Deligne} that the absolute value of \eqref{E:c1c2sum} is
	\begin{equation*}%\label{eqn:cuspbound651}
		\begin{cases}
			2|c_1(n)+c_2(n)|\leq 4d(n)\sqrt n&\text{if }n\equiv 1\pmod{3},\\
			2\sqrt{2}|c_2(n)-c_1(n)|\leq 4\sqrt{2}d(n)\sqrt n&\text{if }n\equiv 2\pmod{3},\\
			0&\text{if }3\mid n.
		\end{cases}
	\end{equation*}

	We next look at the Eisenstein series part. For ease of notation, we abbreviate $\nu_p:=\ord_p(n)$.  For a prime $p$ and $\nu\in\N_0$ we define
\begin{equation*}%\label{eqn:Fpnudef}
	F_p(\nu) :=
	\begin{cases}
		\frac{2^{\nu+2}-1}{3} & \text{if $p=2$},\\
		\frac{p^{\nu+1}-1}{p-1} & \text{if $p\equiv \pm1\Pmod{12}$},\\
		\frac{p^{\nu+1}+1}{p+1} & \text{if $p\equiv \pm5\Pmod{12}$, $2\mid\nu$},\\
		\frac{p^{\nu+1}-1}{p+1} & \text{if $p\equiv\pm5\Pmod{12}$, $2\nmid\nu$}.
	\end{cases}
\end{equation*}
Note that for $p$ odd
\begin{equation}\label{eqn:abseval}
F_p(\nu) = \left|\frac{1-\left(\left(\frac3p\right)p\right)^{\nu+1}}{1-\left(\frac3p\right)p}\right|.
\end{equation}
A direct calculation using \Cref{lem:E65coeff} with $\nu=\nu_2$ and $\mu=0$ then shows that for $n\equiv 1\pmod{3}$ the $n$-th Fourier coefficient of $f_1$ is non-zero if
\begin{equation}\label{eqn:G1def}
	G_1(n) := \prod_{p\mid n} \frac{F_p\left(\nu_p\right)}{(\nu_p+1)p^{\frac{\nu_p}{2}}} > \frac{4}{3}.
\end{equation}
We therefore next determine those $n$ for which $G_1(n)\leq \frac{4}{3}$. For this, for each prime $p$ and $\nu\in\N$ we determine certain constants $\mathcal C_p(\nu)$ for which
\begin{equation*}
	G_1\left(p^{\nu}\right)=\frac{F_p(\nu)}{(\nu+1)p^{\frac{\nu}{2}}} \ge \mathcal C_p(\nu).
\end{equation*}

We first consider the case $p\ne2$.  Bounding against the worst case for $F_{p}(\nu)$, we have
\begin{equation*}%\label{eqn:Gpnu}
	G_1\left(p^{\nu}\right) \ge \frac{p^{\nu+1}-1}{(\nu+1)(p+1)p^{\frac{\nu}{2}}}.
\end{equation*}

A direct calculation shows that
\begin{equation*}%\label{eqn:gnudef}
	g_\nu(x) := \frac{x^{\nu+1}-1}{(x+1)x^{\frac{\nu}{2}}}
\end{equation*}
is increasing for $x>0$, so $g_{\nu}(x)>\aa(\nu+1)$ implies that for $p\geq x$ we have
\begin{equation}\label{eqn:G1boundg}
G_1\left(p^{\nu}\right)\geq \frac{g_{\nu}(p)}{\nu+1} \geq \frac{g_{\nu}(x)}{\nu+1}>\aa.
\end{equation}
 Define for $x\ge3$ and $\aa\in\R_{\geq 1}$
\begin{equation*}%\label{eqn:fanudef}
	f_{\aa,\nu}(x) := x^{\nu+1} - 1 - \aa(\nu+1)(x+1)x^{\frac{\nu}{2}}.
\end{equation*}
Using induction on $\nu$, one can show that if $f_{\aa,\nu}(x)\ge0$, then $f_{\aa,\nu+j}(x)\geq 0$ for all $j\in\N_0$. Hence if $f_{\aa,\mu}(x)\geq 0$ for some $\mu\in\N$, then $g_{\nu}(x)\geq \aa(\nu+1)$ for all $\nu\in\N_0$ with $\nu\geq \mu$. Combining with \eqref{eqn:G1boundg}, we see that if $f_{\aa,\mu}(x)\geq 0$, then for all $p\geq x$ and $\nu\in\N_0$ with $\nu\geq \mu$
\begin{equation}\label{eqn:fG1bound}
G_1\left(p^{\nu}\right)\geq \frac{g_{\nu}(p)}{\nu+1}\geq \frac{g_{\nu}(x)}{\nu+1}>\aa.
\end{equation}
Directly computing
\begin{equation}\label{eqn:f5nux}
	f_{2.1,1}(20)\ge0,\qquad f_{2.1,2}(8)\ge0,\qquad f_{2.1,3}(5)\geq 0,
\end{equation}
we conclude from \eqref{eqn:fG1bound} that
\begin{align*}
G_1\left(p^{\nu}\right)&\geq  2.1\text{ for }p\geq 23,\ \nu\in\N, & G_1\left(p^{\nu}\right)&\geq 2.1\text{ for }p\in\{11,13,17,19\},\nu\geq 2,\\
G_1\left(p^{\nu}\right)&\geq 2.1\text{ for }p\in\{5,7\},\nu\geq 3.
\end{align*}
One also directly checks that for $\nu\geq 6$ we have
\[
G_1(2^\nu)> \frac{2\sqrt{5}}{3}.
\]
 We therefore conclude that if $G_1(n)\leq \frac{4}{3}$, then
\[
n=\prod_{p\leq 19} p^{\nu_p}
\]
with $0\leq\nu_2\leq 5$, $0\leq \nu_5,\nu_7\leq 2$, and $0\leq \nu_p\leq 1$ for $11\leq p\leq 19$. Checking all cases explicitly with a computer, we conclude that $G_1(n)>\frac{4}{3}$ for $n>1120$. It was verified with a computer that for $n\leq1120$ the $n$-th Fourier coefficient of $f_1$ is positive, so we conclude the claim for $n\equiv 1\pmod{3}$.

	Next assume $3\mid n$. 	In this case, we have $c_1(n)=c_2(n)=0$ by \Cref{L:id}, so the $n$-th Fourier coefficient of $f_1$ agrees with the $n$-th Fourier coefficient of $\mathcal E_1$, and we only need to show that the $n$-th Fourier coefficient of $\mathcal E_1$ does not vanish. Then the first factor in Lemma \ref{lem:E65coeff}, plugging in $\nu=\nu_2$ and $\mu=\nu_3$ and abbreviating $m:=\prod_{\substack{p\mid n\\ p\nmid 6}} p^{\nu_p}$, equals
		\begin{equation}\label{E:First}
			-1 + (-1)^{\nu_2+\nu_3+1} 3^{\nu_3}\left(\frac{-3}m\right) + (-1)^{\nu_2+\nu_3+1} 2^{\nu_2+2} \left(\frac{-1}m\right) - 2^{\nu_2+2} 3^{\nu_3} \left(\frac 3m\right).
		\end{equation}
	By the triangle inequality, we may bound the absolute value of $\eqref{E:First}$ from below by
	\begin{equation*}
	2^{\nu_2+2}3^{\nu_3} - 1 - 3^{\nu_3} - 2^{\nu_2+2} = 2^{\nu_2+2}(3^{\nu_3}-1) -1 - 3^{\nu_3} \ge 4(3^{\nu_3}-1) -1 -3^{\nu_3} = 3^{\nu_3+1}-5 > 0.
	\end{equation*}
	So in particular this factor does not vanish. The other factors in Lemma \ref{lem:E65coeff} satisfy
	\[
	\frac{1-\left(\left(\frac{3}{p}\right) p\right)^{\nu_p+1}}{1-\left(\frac{3}{p}\right) p}\neq 0,
	\]
	so the $n$-th Fourier coefficient of $\mathcal E_1$ does not vanish for $3\mid n$.
\end{proof}

Let $a_1(n):=C_{1^{-1}2^{10}3^{-1}4^{-4}}(n)$. Using Lemmas \ref{lem:E65coeff} and \ref{lem:f65identity} and then simplifying with \Cref{L:id}, we obtain the following.
\begin{lem}\label{lem:vanishf651}
	For $n\equiv 2\pmod{3}$, we have $a_1(n)=0$ if and only if $c_1(n)=0$.
\end{lem}
\Cref{lem:vanishf651} and \Cref{prop:f3612vanish} directly give the easy direction of \Cref{thm:L65} for $a_1(n)$.
\begin{lem}\label{lem:conjf651easy}
	If $n\equiv 2\pmod{3}$ and if there exists $p\equiv 3\pmod{4}$ with $\ord_p(n)$ odd, then $a_1(n)=0$.
\end{lem}
We may now conclude the first half of Theorem \ref{thm:L65}, using Corollary \ref{cor:f651-not2mod3} and \Cref{prop:f3612vanish}.
\begin{thm}\label{T:L651}
	We have $a_1(n)=0$ if and only if $n\equiv 2\pmod{3}$ and there exists a prime $p\equiv 3\pmod{4}$ for which $\ord_p(n)$ is odd.
\end{thm}

\subsection{The case $1^72^{-2}3^{-1}$}

Since
\[
	f_2(z):=\frac{\eta^7(z)}{\eta^2(2z)\eta(3z)}
\]
is not a cusp form, we define a corresponding Eisenstein series
\begin{align*}
	&\mathcal E_2(z):=1+\sum_{n\ge1}\left( -\frac{1}{2}\sum_{d\mid n} \left(\frac{12}{d}\right)d -\frac{1}{2}\sum_{d\mid n} \left(\frac{-3}{\frac{n}{d}}\right)\left(\frac{-4}{d}\right) d + \sum_{d\mid n} \left(\frac{-4}{\frac{n}{d}}\right)\left(\frac{-3}{d}\right)d+\hspace{-.05cm}\sum_{d\mid n} \left(\frac{12}{\frac{n}{d}}\right)d\right.\\
	&+\frac{3}{2}\sum_{d\mid \frac{n}{3}} \left(\frac{12}{d}\right)d -\frac{9}{2}\sum_{d\mid \frac{n}{3}} \left(\frac{-3}{\frac{n}{3d}}\right)\left(\frac{-4}{d}\right) d +3 \sum_{d\mid \frac{n}{3}} \left(\frac{-4}{\frac{n}{3d}}\right)\left(\frac{-3}{d}\right)d -9\sum_{d\mid \frac{n}{3}} \left(\frac{12}{\frac{n}{3d}}\right)d+\sum_{d\mid \frac{n}{2}} \left(\frac{12}{d}\right)d \\
	& -\!\sum_{d\mid \frac{n}{2}} \left(\frac{-3}{\frac{n}{2d}}\right)\!\left(\frac{-4}{d}\right)\!d +4 \sum_{d\mid \frac{n}{2}} \left(\frac{-4}{\frac{n}{2d}}\right)\!\!\left(\frac{-3}{d}\right)\!d -4\sum_{d\mid \frac{n}{2}} \left(\frac{12}{\frac{n}{2d}}\right)\!d- 3\sum_{d\mid \frac{n}{6}}\left(\frac{12}{d}\right)\!d -9\sum_{d\mid \frac{n}{6}} \!\left(\frac{-3}{\frac{n}{6d}}\right)\!\!\left(\frac{-4}{d}\right)\!d\\
	&\left.   +12\sum_{d\mid \frac{n}{6}} \left(\frac{-4}{\frac{n}{6d}}\right)\left(\frac{-3}{d}\right)d +36\sum_{d\mid \frac{n}{6}} \left(\frac{12}{\frac{n}{6d}}\right)d \right)q^n.
\end{align*}
A direct calculation shows the following.
\begin{lem}\label{lem:E652coeff}
	Suppose that $\nu_2,\nu_3\in\N_0$ and $m\in\N$ with $\gcd(m,6)=1$. Then the $2^{\nu_2}3^{\nu_3} m$-th Fourier coefficient of $\mathcal E_2$ equals
	\begin{align*}
		&\bigg(-\frac12-\frac{3^{{\nu_3}}}{2}\left(\frac{-3}{2^{{\nu_2}}m}\right)\left(\frac{-1}{3^{\nu_3}}\right)+ 2^{\nu_2}\left(\frac{-3}{2^{\nu_2}}\right)\left(\frac{-1}{3^{\nu_3} m}\right) + 2^{\nu_2} 3^{\nu_3} \left(\frac3m\right)\\
		&\ \ + \mathbbm{1}_{3\mid n} \left(\frac32 - \frac12 3^{{\nu_3}+1}\left(\frac{-3}{2^{{\nu_2}}m}\right) \left(\frac{-1}{3^{{\nu_3}-1}}\right) + 3 \cdot 2^{\nu_2} \left(\frac{-3}{2^{\nu_2}}\right) \left(\frac{-1}{3^{{\nu_3}-1} m}\right) - 2^{\nu_2} 3^{{\nu_3}+1} \left(\frac3m\right)\right)\\
		&\ \ + \mathbbm{1}_{2\mid n} \left(1-3^{\nu_3}\left(\frac{-3}{2^{{\nu_2}-1}m}\right)\left(\frac{-1}{3^{\nu_3}}\right) + 2^{{\nu_2}+1} \left(\frac{-3}{2^{{\nu_2}-1}}\right)\left(\frac{-1}{3^{\nu_3} m}\right) - 2^{{\nu_2}+1}3^{{\nu_3}}\left(\frac3m\right)\right)\\
		&\ \ + \mathbbm{1}_{6\mid n} \left(-3-3^{{\nu_3}+1} \left(\frac{-3}{2^{{\nu_2}-1}m}\right) \left(\frac{-1}{3^{{\nu_3}-1}}\right) + 2^{{\nu_2}+1} 3\left(\frac{-3}{2^{{\nu_2}-1}}\right) \left(\frac{-1}{3^{{\nu_3}-1} m}\right)+2^{{\nu_2}+1}3^{{\nu_3}+1}\left(\frac3m\right)\right)\bigg)\\
		&\quad\quad\times \prod_{p\mid m} \frac{1-\left(\left(\frac3p\right)p\right)^{\ord_p(m)+1}}{1-\left(\frac3p\right)p}.
	\end{align*}
\end{lem}
We directly obtain the following lemma.
\begin{lem}\label{lem:E652}
We have
\begin{align*}
\mathcal E_2&=-\frac{1}{4}E_{2,\chi_{1},\chi_{12}}-\frac{1}{4}E_{2,\chi_{-3},\chi_{-4}}+\frac{1}{2}E_{2,\chi_{-4},\chi_{-3}}+\frac{1}{2}E_{2,\chi_{12},\chi_{1}}\\
&\hspace{.5cm}+\frac{3}{4}E_{2,\chi_{1},\chi_{12}}\big|V_3-\frac{9}{4}E_{2,\chi_{-3},\chi_{-4}}\big|V_3+\frac{3}{2}E_{2,\chi_{-4},\chi_{-3}}\big|V_3-\frac{9}{2}E_{2,\chi_{12},\chi_{1}}\big|V_3\\
&\hspace{.5cm}+\frac{1}{2}E_{2,\chi_{1},\chi_{12}}\big|V_2-\frac{1}{2}E_{2,\chi_{-3},\chi_{-4}}\big|V_2+2E_{2,\chi_{-4},\chi_{-3}}\big|V_2-2E_{2,\chi_{12},\chi_{1}}\big|V_2\\
&\hspace{.5cm}-\frac{3}{2}E_{2,\chi_{1},\chi_{12}}\big|V_6-\frac{9}{2}E_{2,\chi_{-3},\chi_{-4}}\big|V_6+6E_{2,\chi_{-4},\chi_{-3}}\big|V_6+18E_{2,\chi_{12},\chi_{1}}\big|V_6.
\end{align*}
\end{lem}
Using Lemmas \ref{lem:eta}, \ref{lem:E652}, \ref{lem:Eisenstein}, and \ref{lem:valence}, we obtain an identity for $f_2$.
\begin{lem}\label{lem:f652}
We have
\begin{equation*}%\label{eqn:f652identity}
f_2
%=\frac{\eta(z)^7}{\eta(2z)^2\eta(3z)}
=\mathcal E_2 -4\Big(1-\sqrt{2}i\Big)g_1 -4\Big(1+\sqrt{2}i\Big)g_2.
\end{equation*}
\end{lem}
We next classify those $n$ for which $a_2(n):=C_{1^{7}2^{-2}3^{-1}}(n)=0$.
\begin{thm}\label{thm:c(1,7),(2,-2),(3,-1)}
	We have $a_2(n)=0$ if and only if $n\equiv 2\pmod{3}$ and there exists $p\equiv 3\pmod{4}$ for which $\ord_p(n)$ is odd.
\end{thm}
\begin{proof}%[Proof of \Cref{thm:c(1,7),(2,-2),(3,-1)}]
	As above, we write $n=2^{\nu_2} 3^{\nu_3}m$ with $\gcd(m,6)=1$. For ${\nu_3}=0$, Lemma \ref{lem:E652coeff} implies that the $2^{{\nu_2}}m$-th Fourier coefficient of $\mathcal E_2$ is $\prod_{p\mid m} \frac{1-((\frac3p)p)^{\ord_p(m)+1}}{1-(\frac3p)p}$ times  (note  $(\frac{-3}{2})=-1$)
	\begin{multline}\label{eqn:E2constevalmu0}
		-\frac12\left(1+\left(\frac{-3}{n}\right)\right)+ 2^{\nu_2}\left(\frac{-1}{m}\right)\left(\frac{-3}{2^{\nu_2}}\right)\left( 1+\left(\frac{-3}{n}\right)\right)\\
		+ \mathbbm{1}_{2\mid n}\left(1+\left(\frac{-3}{n}\right) + 2^{{\nu_2}+1} \left(\frac{-3}{2^{{\nu_2}-1}}\right)\left(\frac{-1}{m}\right)\left(1+\left(\frac{-3}{n}\right)\right)\right).
	\end{multline}
	In particular, if $n\equiv 2\pmod{3}$, then $(\frac{-3}{n})=-1$ and we see that the Fourier coefficient of $\mathcal E_2$ vanishes. Thus, for $n\equiv 2\pmod{3}$, Lemma \ref{lem:f652} implies that
	\begin{equation*}
		a_2(n) = -4\left(c_1(n)+c_2(n)\right) +4\sqrt{2}i\left(c_1(n)-c_2(n)\right).
	\end{equation*}
	Using Lemma \ref{L:id}, one easily obtains that for $n\equiv 2\pmod{3}$
	\[
		a_2(n)=8\sqrt{2}ic_1(n).
	\]
	Therefore $a_2(n)=0$ if and only if $c_1(n)=0$ for $n\equiv 2\pmod{3}$.
 By Proposition \ref{prop:f3612vanish}, since $n\equiv 2\pmod{3}$ (and hence $3\nmid n$ in particular) this occurs if and only if there exists a prime $p\equiv 3\pmod{4}$ for which $\ord_p(n)$ is odd. This gives the claim for $n\equiv 2\pmod{3}$.

	For $n\equiv 1\pmod{3}$, we have $(\frac{-3}{n})=1$, so, after simplifying \eqref{eqn:E2constevalmu0}, \Cref{lem:E652coeff} implies that the $2^{{\nu_2}}m$-th Fourier coefficient of $\mathcal E_2$ is
	\[
		\prod_{p\mid m} \frac{1-\left(\left(\frac3p\right)p\right)^{\ord_p(m)+1}}{1-\left(\frac3p\right)p}\left( -1+2^{{\nu_2}+1}\left(\frac{-1}{m}\right)\left(\frac{-3}{2^{{\nu_2}}}\right)+2\mathbbm{1}_{2\mid n}+2^{{\nu_2}+2}\mathbbm{1}_{2\mid n}\left(\frac{-3}{2^{{\nu_2}-1}}\right)\left(\frac{-1}{m}\right)\right).
	\]
	We then note that
	\begin{multline*}
		\left| -1+2^{{\nu_2}+1}\left(\frac{-1}{m}\right)\left(\frac{-3}{2^{{\nu_2}}}\right)+\mathbbm{1}_{2\mid n}\left(2+2^{{\nu_2}+2} \left(\frac{-3}{2^{{\nu_2}-1}}\right)\left(\frac{-1}{m}\right)\right)\right|\\
		=
		\begin{cases}
			\left|2\left(\frac{-1}{m}\right)-1\right|\geq 1 &\text{if }2\nmid n,\\
			\left|\left(\frac{-1}{m}\right)\left(\frac{-3}{2^{{\nu_2}}}\right)\left(2^{{\nu_2}+1}-2^{{\nu_2}+2}\right) +1\right|\geq 2^{{\nu_2}+1}-1&\text{if }2\mid n.
		\end{cases}
	\end{multline*}
	Combining with \eqref{eqn:abseval}, the absolute value of the $2^{{\nu_2}}m$-th Fourier coefficient of $\mathcal E_2$ is bounded from below by
	\begin{equation}\label{eqn:Eabs}
		\left(2^{{\nu_2}+1}-1\right) \prod_{p\mid m} F_p(\ord_p(m)).
	\end{equation}
Plugging \Cref{L:id} in to evaluate the cuspidal part of \Cref{lem:f652}, we conclude for $n\equiv 1\pmod{3}$ that $a_2(n)\neq 0$ if the expression in \eqref{eqn:Eabs} is greater than $8|c_1(n)|$.
	Using \Cref{thm:Deligne}, the absolute value of the Fourier coefficient of the cuspidal part is bounded from above by $8d(n)\sqrt{n}$.
	We conclude that $a_2(n)\neq 0$ if
\[
		\left(2^{{\nu_2}+1}-1\right) \prod_{p\mid m} F_p(\ord_p(m))> 8d(n)\sqrt{n}.
\]
Defining
\[
	G_2(n):=\frac{2^{\ord_2(n)+1}-1}{(\ord_2(n)+1)2^{\frac{\ord_2(n)}{2}}} \prod_{\substack{p\mid n\\p\neq2}} \frac{F_{p}(\ord_p(n))}{(\ord_p(n)+1)p^{\frac{\ord_p(n)}{2}}}
\]
and rearranging, we conclude that if $G_2(n)> 8$ then $a_2(n)\neq 0$.
By construction, $G_2$ is multiplicative, and $G_2(n)=G_1(n)$ for odd $n\in\N$ (see \eqref{eqn:G1def}), so by \eqref{eqn:fG1bound} we can obtain a bound on $\ord_p(n)$ for $p$ odd after evaluating $f_{\aa,\nu}(x)$. As in \eqref{eqn:f5nux}, we have
\begin{align*}
 f_{10,1}(402)&\geq 0,& f_{10,2}(31)&\geq 0,&f_{10,3}(13)&\geq 0,&f_{10,4}(8)&\geq 0,\\
f_{10,5}(6)&\geq 0,&f_{10,6}(5)&\geq 0.
\end{align*}
Thus \eqref{eqn:fG1bound} yields
\begin{align}
\nonumber G_2\left(p^{\nu}\right)&>  10\text{ for }p\geq 402,\ \nu\in\N,& G_2\left(p^{\nu}\right)&> 10\text{ for }31\leq p< 402,\ \nu\geq 2,\\
\nonumber G_2\left(p^{\nu}\right)&> 10\text{ for }13\leq p<31,\ \nu\geq 3,& G_2\left(p^{\nu}\right)&> 10\text{ for }p=11,\ \nu\geq 4,\\
\label{eqn:f10nux} G_2\left(p^{\nu}\right)&> 10\text{ for }p=7,\nu\geq 5,&G_2\left(p^{\nu}\right)&> 10\text{ for }p=6,\ \nu\geq 6.
\end{align}
\rm
Moreover, if $\#\{p\text{ prime}:p\|n\}\geq 7$, then bounding against the worst-case choice of $7$ primes gives $G_2(n)\geq 8$. Hence we conclude from \eqref{eqn:f10nux} and a direct computation of $G_2(2^{\nu})$ that
\begin{equation}\label{eqn:nexpandG2}
n=2^{\nu_2}\prod_{3<p\leq 401} p^{\nu_p}
\end{equation}
with $0\leq \nu_2\leq 12$, $0\leq \nu_5\leq 5$, $0\leq \nu_7\leq 4$,  $0\leq \nu_{11}\leq 3$, $0\leq \nu_{13}\leq 2$, $0\leq \nu_p\leq 1$ for $p\geq 17$, and $\#\{p: \nu_p=1\}\leq 6$.
We used a computer to evaluate $G_2(n)$ for every such $n\equiv 1\pmod{3}$ of the type \eqref{eqn:nexpandG2}, and find that $G_2(n)>8$ for $n>309400$ with $n\equiv 1\pmod{3}$. It was verified with a computer (running code that completed in a few hours on a standard desktop computer) that $C_{1^72^{-2}3^{-1}}(n)\neq 0$ for $n\leq 309400$ with $n\equiv 1\pmod{3}$, yielding the claim for $n\equiv 1\pmod{3}$.

For $3\mid n$, \Cref{L:id} implies that the Fourier coefficient of the cusp form appearing on the right-hand side of \Cref{lem:f652} vanishes, so by \Cref{lem:f652}, $a_2(n)=0$ if and only if the $n$-th Fourier coefficient of $\mathcal E_2$ vanishes. By Lemma \ref{lem:E652coeff}, this is the case if and only if
	\begin{align}
		\nonumber		&-\frac12-\frac{3^{{\nu_3}}}{2}\left(\frac{-3}{2^{\nu_2}m}\right)\left(\frac{-1}{3^{\nu_3}}\right)+ 2^{\nu_2}\left(\frac{-3}{2^{\nu_2}}\right)\left(\frac{-1}{3^{\nu_3} m}\right) + 2^{\nu_2} 3^{\nu_3} \left(\frac3m\right)\\
		\nonumber		&\qquad+ \left(\frac32 - \frac12 3^{{\nu_3}+1}\left(\frac{-3}{2^{\nu_2}m}\right) \left(\frac{-1}{3^{{\nu_3}-1}}\right) + 3 \cdot 2^{\nu_2} \left(\frac{-3}{2^{\nu_2}}\right) \left(\frac{-1}{3^{{\nu_3}-1} m}\right) - 2^{\nu_2} 3^{{\nu_3}+1} \left(\frac3m\right)\right)\\
		\label{eqn:leadingE652}		&\qquad+ \mathbbm{1}_{2\mid n} \left(1-3^{\nu_3}\left(\frac{-3}{2^{\nu_2-1}m}\right)\left(\frac{-1}{3^{\nu_3}}\right) + 2^{\nu_2+1} \left(\frac{-3}{2^{\nu_2-1}}\right)\left(\frac{-1}{3^{\nu_3} m}\right) - 2^{\nu_2+1}3^{{\nu_3}}\left(\frac3m\right)\right.\\
		\nonumber
		&\qquad\left.  -3-3^{{\nu_3} +1} \left(\frac{-3}{2^{\nu_2-1}m}\right) \left(\frac{-1}{3^{{\nu_3}-1}}\right) + 2^{\nu_2+1} 3\left(\frac{-3}{2^{\nu_2-1}}\right) \left(\frac{-1}{3^{{\nu_3}-1} m}\right) + 2^{\nu_2+1}3^{{\nu_3}+1}\left(\frac3m\right)\right)
	\end{align}
	vanishes. If $2\nmid n$, then $\nu_2=0$ and \eqref{eqn:leadingE652} simplifies as
	\begin{equation*}
		\left(1-2\left(\frac{-1}{n}\right)\right) +3^{{\nu_3}}\left(\frac{3}{m}\right)\left(\left(\frac{-1}{n}\right) - 2\right).
	\end{equation*}
	This vanishes if and only if
	\begin{equation}\label{eqn:3contradiction}
		3^{{\nu_3}}\left(\frac{3}{m}\right)\left(\left(\frac{-1}{n}\right) - 2\right)=-1+2\left(\frac{-1}{n}\right).
	\end{equation}
	Since $3$ divides the left-hand side of \eqref{eqn:3contradiction} (we have ${\nu_3}\geq 1$ because $3\mid n$), it must divide the right-hand side, which can only occur if $(\frac{-1}{n})=-1$, in which case the right-hand side equals $-3$. But then \eqref{eqn:3contradiction} simplifies to
	$3^{{\nu_3}}(\frac{3}{m})(-3)=-3$,
which can only occur if ${\nu_3}=0$, leading to a contradiction.
Hence for $3\mid n$ and $2\nmid n$ we conclude that $a_2(n)\neq0$.

Finally suppose that $2\mid n$. Combining  terms with the same Legendre symbols, \eqref{eqn:leadingE652} becomes (note that $(\frac{-3}{2})=(\frac{-1}{3})=-1$)
	\begin{equation}\label{eqn:3midnsimp}
		-1 +2^{{\nu_2}+1} 3^{\nu_3}\left( -\frac{1}{2^{{\nu_2}+1}} \left(\frac{-3}{2^{{\nu_2}}m}\right)\left(\frac{-1}{3^{\nu_3}}\right)+ \frac{1}{3^{{\nu_3}}} \left(\frac{-3}{2^{\nu_2}}\right)\left(\frac{-1}{3^{\nu_3} m}\right)+  \left(\frac3m\right)\right).
	\end{equation}
Since ${\nu_2},{\nu_3}\in\N$, the absolute value of \eqref{eqn:3midnsimp} is bounded from below by
	\[
		2^{{\nu_2}+1} 3^{\nu_3}\left( 1-\frac{1}{4}-\frac{1}{3} \right) -1 = \frac{5}{12} 2^{{\nu_2}+1}{3^{{\nu_3}}} - 1\geq \frac{5}{12}\cdot 4\cdot 3 -1=4.
	\]
	We conclude that \eqref{eqn:leadingE652} does not vanish, and hence
 $C_{1^72^{-2}3^{-1}}(n)\neq 0$, for all $n$ with $3\mid n$.
\end{proof}
\section{Proof of \Cref{thm:L133}}\label{sec:L133}
\subsection{The case $1^22^34^{-2}$}
In this subsection, we prove the claimed evaluation of the first set appearing in \Cref{thm:L133}.
\begin{prop}\label{prop:L1331}
We have
\[
S_{1^22^34^{-2}}=\left\{n\in\N: n=4^k(8m+7)\text{ for some }k,m\in\N_0\right\}.
\]
\end{prop}
We first obtain a formula for
\[
	f_3(z):=\frac{\eta^2(z)\eta^3(2z)}{\eta^2(4z)}.
\]

We have the following lemma.
\begin{lem}\label{lem:72}
	We have
	\[
		f_3=\mathcal{H}_{1,2}\big| U_2\big|\left(12S_{4,0}-4S_{4,1}-4S_{4,2}+12 S_{4,3}\right).
	\]
\end{lem}

\begin{proof}
By \cite[Proposition 1.41]{O}, \cite[Theorem 1.60]{O}, Lemmas \ref{lem:halfintmult}, and \ref{lem:eta}, we have
\[
f_3(z)=\frac{\eta^4(z)}{\eta^2(2z)} \Theta(z)\in M_{\frac32}(\Gamma_0(8)).
\]
 For the right-hand side, observe that by \Cref{lem:Hell1ell2}, $\mathcal H_{1,2}\in M_\frac32(\Gamma_0(8),\chi_8)$. Applying $U_2$ gives, by \Cref{lem:operatorshalf} (1), an element of $M_\frac32(\Gamma_0(8))$. Finally, by \Cref{lem:operatorshalf} (2), $S_{4,a}$ ($a\in\{1,2,3\}$) gives an element of $M_\frac32(\Gamma_0(16))$. Thus both sides are modular forms of weight $\frac32$ on $\Gamma_{0}(16)$. So we have to check $3$ Fourier coefficients, which were checked with a computer.
\end{proof}

We are now ready to prove \Cref{prop:L1331}.
\begin{proof}[Proof of \Cref{prop:L1331}]
By Lemma \ref{lem:72}, we have $C_{1^22^34^{-2}}(n)=0$ if and only if the $n$-th Fourier coefficient of $\mathcal{H}_{1,2}\big|U_2$ vanishes. However, by \cite[Lemma 4.1]{BK}, we have $\Theta^3 = 12\mathcal{H}_{1,2}\big|U_2$, so $C_{1^22^34^{-2}}(n)=0$ if and only if $r_{(1,1,1)}(n)=0$.  By Legendre's three-square theorem \cite{Legendre}
\begin{equation}\label{eqn:LegendreThreeSquare}
	r_{(1,1,1)}(n)=0 \Leftrightarrow n=4^k(8m+7)\text{ for some }k,m\in\N_0.%\qedhere
\end{equation}
This is the claim.
\end{proof}
\subsection{The case $1^62^{-3}$}
In this subsection, we prove the claimed identity for the other set appearing in \Cref{thm:L133}.
\begin{prop}\label{prop:L1332}
We have
\[
S_{1^62^{-3}}=\left\{n\in\N: n=4^k(8m+7)\text{ for some }k,m\in\N_0\right\}.
\]
\end{prop}

\begin{proof}
Using $\frac{\eta(z)^2}{\eta(2z)}=\sum_{n\in\Z} (-1)^nq^{n^2}$ (see \cite[Theorem 1.60]{O}), one directly obtains
\[
\frac{\eta^6(z)}{\eta^3(2z)}=\left(\sum_{n\in\Z} (-1)^n q^{n^2}\right)^3=\sum_{n_1,n_2,n_3\in\Z}(-1)^{n_1+n_2+n_3}q^{n_1^2+n_2^2+n_3^2}= \sum_{n\ge0} (-1)^n r_{(1,1,1)}(n)q^{n}.
\]
By \eqref{eqn:LegendreThreeSquare}, we have
\[
	r_{(1,1,1)}(n)=0 \Leftrightarrow n=4^k(8m+7)\text{ for some }k,m\in\N_0.\qedhere
\]
\end{proof}
	

\begin{thebibliography}{99}
\bibitem{Legendre}A. Legendre, \begin{it}Essai sur la th\'eorie des nombres\end{it}, Paris An VI (1797), p. 202 and p. 398--399.
\end{thebibliography}
\end{document}